\documentclass[11pt,reqno,a4paper]{amsart}
\usepackage{amsmath,amsthm,amsfonts,amssymb,xcolor,enumitem,comment,hyperref}
\usepackage[numbers,sort&compress]{natbib}
\usepackage{graphicx}
\usepackage{setspace}
\usepackage{mathrsfs}
\usepackage{hyperref}
\usepackage[letterpaper,margin=1.3in]{geometry}

\usepackage{comment}

\newcommand{\R}{\mathbb{R}}

\renewcommand{\H}{\mathbb{H}}

\newcommand{\g}{\gamma}

\renewcommand{\epsilon}{\varepsilon}
\renewcommand{\hat}{\widehat}
\renewcommand{\tilde}{\widetilde}

\def\Xint#1{\mathchoice
   {\XXint\displaystyle\textstyle{#1}}%
   {\XXint\textstyle\scriptstyle{#1}}%
   {\XXint\scriptstyle\scriptscriptstyle{#1}}%
   {\XXint\scriptscriptstyle\scriptscriptstyle{#1}}%
   \!\int}
\def\XXint#1#2#3{{\setbox0=\hbox{$#1{#2#3}{\int}$}
     \vcenter{\hbox{$#2#3$}}\kern-.5\wd0}}

\def\dashint{\Xint-}

\newtheorem{theorem}{Theorem}[section]
\newtheorem{proposition}[theorem]{Proposition}
\newtheorem{lemma}[theorem]{Lemma}

\theoremstyle{definition}
\newtheorem{definition}[theorem]{Definition}

\theoremstyle{remark}
\newtheorem{remark}[theorem]{Remark}

\numberwithin{equation}{section}

\newcommand{\diam}{\mathop\mathrm{diam}\nolimits}

\title[Higher Order Whitney and Lusin for Curves in $\mathbb{H}$]{Higher Order Whitney Extension and Lusin Approximation for Horizontal Curves in the Heisenberg Group}

\author[Andrea Pinamonti]{Andrea Pinamonti}
\address{ Department of Mathematics, University of Trento, Via Sommarive 14, 38123 Povo (Trento), Italy}
\email{Andrea.Pinamonti@unitn.it}

\author[Gareth Speight]{Gareth Speight}
\address{Department of Mathematical Sciences, University of Cincinnati, 2815 Commons Way, Cincinnati, OH 45221, United States}
\email{Gareth.Speight@uc.edu}

\author[Scott Zimmerman]{Scott Zimmerman}
\address{Department of Mathematics, 
The Ohio State University at Marion,
1465 Mt Vernon Ave, Marion, OH 43302, United States}
\email{Zimmerman.416@osu.edu}

\subjclass{58C25, 53C17.}

\allowdisplaybreaks

\begin{document}

\begin{abstract}
In the setting of horizontal curves in the Heisenberg group, we prove a $C^{m,\omega}$ finiteness principle, a $C^{m,\omega}$ Lusin approximation result, a $C^{\infty}$ Whitney extension result, and a $C^{\infty}$ Lusin approximation result. Combined with previous work, this completes the study of Whitney extension and Lusin approximation for horizontal curves of class $C^{m}$, $C^{m,\omega}$, and $C^{\infty}$ in the Heisenberg group.
\end{abstract}

\keywords{Finiteness principle; Whitney extension theorem; Heisenberg group}

\date{\today}

\maketitle



\section{Introduction}

The classical Whitney extension theorem \cite{Whitney} characterizes those collections of real-valued continuous functions $F=(F^{k})_{|k|\leq m}$, defined on a compact set $K\subset \mathbb{R}^{n}$, which can be extended to a $C^{m}$ function $f$ such that the derivatives satisfy $D^{k}f|_{K}=F^{k}$ for multiindices $|k|\leq m$. The key condition is that the collection $F$ must form a Whitney field. This encodes the fact that the maps in the collection must be related as dictated by Taylor's theorem. There have been versions of Whitney's result for mappings with varying regularity \cite{Bierstone, FefferWhit, FefferWhitFull, FefferWhitLinear, FefferSummary} and between different spaces \cite{FSS01, FSS03, FSS03b, Sig, SaccSiga, PV06, ZimWhitney, ZimFinite}. Whitney extension results have been applied to study rectifiable sets, construct mappings with desired differentiability properties, and prove Lusin approximation results \cite{CapPinSpe, CapPinSpe2, Fed, LedSpe, Speight, Whi35, Liu}. 

In recent years, it has become clear that a large part of geometric analysis, geometric measure theory, and real analysis in Euclidean spaces may be generalized to Carnot groups \cite{Italians, CDPT07, FSS01, FSS03, Mon02, Pansu}. Roughly speaking, a Carnot group is a Lie group (i.e. a smooth manifold equipped with smooth translations) whose Lie algebra (i.e. the space of left invariant vector fields) admits a suitable stratification in which the first layer generates all others. Absolutely continuous curves in a Carnot group which are almost everywhere tangent to the first layer are called horizontal curves and are fundamental to the geometry of the space.

Recently, there have been a number of contributions to the study of Whitney extension theorems and their applications in Carnot groups. The case of real-valued mappings defined on subsets of Carnot groups is quite well understood \cite{PV06}, but the case of mappings whose target is a Carnot group is less so. At present, there are only results for mappings from subsets of $\mathbb{R}$ into Carnot groups. Here, one requires the extension to be a horizontal curve. Existing results include a $C^{1}$ Whitney extension theorem in pliable Carnot groups \cite{Sig} and a $C^{m}$ Whitney extension theorem for $m\geq 1$ both in the Heisenberg group \cite{ZimSpePinWhitney} and in free Carnot groups of step $2$ \cite{H23}.

The present paper continues the work started in \cite{ZimSpePinWhitney}. We focus on the first Heisenberg group $\mathbb{H}$ (see Definition \ref{def-Heis}), which is the simplest and most often studied non-Euclidean Carnot group. While we expect analogues of our results to hold in the Heisenberg group $\mathbb{H}^{n}$ of any dimension, we focus on the first Heisenberg group $\mathbb{H}$ to keep the notation manageable. $\mathbb{H}$ can be viewed in coordinates as $\mathbb{R}^{3}$ with a two-dimensional horizontal distribution. In \cite{ZimSpePinWhitney}, the authors of the present paper proved a Whitney extension theorem for $C^{m}$ horizontal curves in $\mathbb{H}$. The second and third author then proved an analogue for $C^{m,\omega}$ horizontal curves in \cite{ZimSpeWhitney}. Here we recall that a real-valued function $f$ on $\mathbb{R}^n$ is of class $C^{m,\omega}$ if it is $m$-times differentiable and the $m$'th order partial derivatives are uniformly continuous with modulus of continuity $\omega$ (see Definition \ref{defcmw}). For example, if the $m$'th order partial derivatives of $f$ are Lipschitz continuous, then $f$ is of class $C^{m,\omega}$ with $\omega(t) = t$. In the present paper we show that the result of \cite{ZimSpeWhitney} has a number of interesting applications.

Our first result (Theorem \ref{t-finiteness}) establishes a finiteness principle for $C^{m,\omega}$ horizontal curves in the Heisenberg group. This is related to the original Whitney problem: given a compact set $K \subseteq \mathbb{R}^n$ and a continuous function $f:K \to \R$, when is there a function $F \in C^{m,\omega}(\R^n)$ such that $F|_K = f$? Note that we only consider a single continuous function $f$ rather than a collection of possible derivatives. Whitney posed and answered this question in the case $n=1$ in 1934 \cite{Whitney2}. See Theorem~\ref{t-brush} for a formulation due to Brudnyi and Shvartsman \cite{BruShv,BruShv4}. Fefferman then answered this question in full for $n \geq 1$ \cite{FefferWhitFull,FefferWhitLinear}. This has since come to be known as the finiteness principle for real-valued functions on $\mathbb{R}^n$. For curves in the Heisenberg group we ask a related question: given a compact set $K \subset \mathbb{R}$ and a continuous map $\gamma:K \to \mathbb{H}$, when must there be a horizontal $C^{m,\omega}$ curve $\Gamma$ with $\Gamma|_K = \gamma$? Of course, some kind of $C^{m,\omega}$ ``data'' must be imposed on the curve $\gamma$ in order to ensure the existence of a $C^{m,\omega}$ horizontal extension. In \cite{ZimFinite}, the third author proved a version of this finiteness principle but was only able to obtain a $C^{m,\sqrt{\omega}}$ regular extension from $C^{m,\omega}$ initial data. This drop in regularity was due to the method of the proof of the main result of \cite{ZimSpePinWhitney} in the $C^{m}$ setting. Using the $C^{m,\omega}$ version in \cite{ZimSpeWhitney}, we fix this drop in regularity.

Our second result (Theorem \ref{t-Lusin}) establishes an $(m,\omega)$-Lusin property for horizontal curves in $\mathbb{H}$. In general, a map possesses a {\em Lusin property} if it coincides with a smooth map up to a set of arbitrarily small measure. The most well known Lusin property comes from Lusin's theorem, which asserts that, given a measurable map $f\colon \mathbb{R}\to \mathbb{R}$ and $\varepsilon>0$, there exists a continuous map $F\colon \mathbb{R}\to \mathbb{R}$ such that $\mathcal{L}^{1}\{x\in\mathbb{R}:F(x)\neq f(x)\}<\varepsilon$. Other results show that increasing the regularity of the measurable function $f$ allows for more smoothness in the approximating function $F$. For related results in Euclidean spaces, see \cite{Fed, LT94, Menne, Haj, Liu, MZim, Aza1, Isa, Boj}. Lusin properties in Carnot groups were first studied for horizontal curves in the Heisenberg group in \cite{Speight}. There the second author showed directly that every horizontal curve coincides with a $C^{1}$ horizontal curve except for a set of small measure, but the same result does not hold in the Engel group (a step three Carnot group). This was then extended to pliable Carnot groups in \cite{Sig}. In \cite{CapPinSpe2} the first and second authors together with Capolli proved a Lusin approximation by $C^{m}$ horizontal curves in $\mathbb{H}$ by applying the $C^m$ Whitney extension theorem for horizontal curves from \cite{ZimSpePinWhitney}. To prove Theorem \ref{t-Lusin}, we combine the techniques of \cite{CapPinSpe2} with the $C^{m,\omega}$ Whitney extension theorem from \cite{ZimSpeWhitney}.

Our third result (Theorem~\ref{t-Cinfty}) establishes a Whitney extension theorem for $C^{\infty}$ horizontal curves in the Heisenberg group. The classical Whitney extension theorem for $C^{\infty}$ mappings (Theorem \ref{classicalWhitneyInf}) assumes that the jets are $C^{m}$ Whitney fields for every $m\geq 1$, i.e. that the conditions of the $C^{m}$ Whitney extension theorem hold for every $m\geq 1$. It is natural to ask whether it suffices to assume the conditions of the $C^{m}$ Whitney extension theorem for horizontal curves in $\mathbb{H}$ from \cite{ZimSpePinWhitney} for every $m\geq 1$. Such a hypothesis is clearly necessary. Surprisingly, this does not seem to easily yield a $C^{\infty}$ Whitney extension result. Instead, we use the conditions from the $C^{m,\omega}$ Whitney extension theorem for horizontal curves in $\mathbb{H}$ from \cite{ZimSpeWhitney} with $\omega(t)=t$. The key idea here is that, if a curve is $C^{m+1}$, then the $m$'th order derivatives are Lipschitz continuous on the compact set $K$, so the curve is in $C^{m,\omega}$ with $\omega(t)=t$. Hence every $C^{\infty}$ curve is $C^{m,\omega}$ for every $m$.

Our fourth and final result (Theorem \ref{tt-Lusin}) establishes a Lusin approximation result giving conditions under which a horizontal curve in $\mathbb{H}$ can be approximated by a $C^{\infty}$ horizontal curve. This follows by combining Theorem~\ref{t-Cinfty} with the techniques used in the proof of Theorem~\ref{t-Lusin}.

The paper is organized as follows. In Section~\ref{s-prelim}, we introduce notation and the necessary preliminaries for the paper. In Section~\ref{s-finiteness} we establish the finiteness principle for $C^{m,\omega}$ horizontal curves (Theorem~\ref{t-finiteness}). In  Section~\ref{s-lusin} we prove a Lusin approximation theorem for $C^{m,\omega}$ horizontal curves (Theorem~\ref{t-Lusin}). In Section \ref{SectionCinfWhitney} we outline the proof of the $C^\infty$ Whitney extension theorem for horizontal curves (Theorem~\ref{t-Cinfty}). We apply this in Section~\ref{SectionCinfLusin} to obtain a Lusin approximation theorem for $C^{\infty}$ horizontal curves (Theorem \ref{tt-Lusin}).

\bigskip
	
\textbf{Acknowledgements:} A. Pinamonti is a member of {\em Gruppo Nazionale per l'Analisi Matematica, la Probabilit\`a e le loro Applicazioni} (GNAMPA) of the {\em Istituto Nazionale di Alta Matematica} (INdAM).

G. Speight  was supported by a grant from the Simons Foundation (\#576219, G. Speight) and by a fellowship from the Taft Research Center at the University of Cincinnati.

Part of this work was done during a visit of G. Speight to the University of Trento and during a visit of A. Pinamonti to the University of Cincinnati. The authors thank both institutions for the excellent working conditions offered.

\section{Preliminaries}
\label{s-prelim}

\subsection{Whitney Extension in Euclidean Spaces}
\label{s-Euclidean}

We first recall relevant theory related to Whitney extension in Euclidean spaces.

\subsubsection{Notation}
Throughout the paper, $m$ will be a positive integer and $\omega$ will be a (concave) modulus of continuity. By modulus of continuity we mean a continuous, increasing, concave function $\omega:[0,\infty] \to [0,\infty]$ such that $\omega(0)=0$.

Given a quantity $d$, we write $a \lesssim_d b$ to indicate that $a\leq Cb$ where $C=C(d)>0$ is a constant possibly depending on $d$.

For any integer $k>0$, we say that a nonnegative quantity $c(a,b)$ is uniformly $o(|b-a|^k)$ on a set $K$
if, for every $\epsilon > 0$, there is a $\delta>0$ such that $c(a,b) \leq \varepsilon |b-a|^k$ whenever $a,b \in K$ satisfy $|b-a| < \delta$.

\subsubsection{$C^{m}$ Mappings}

\begin{definition}
A \emph{jet} of order $m$ on a set $E\subset \R$ is a collection $F=(F^k)_{k=0}^m$ of continuous, real-valued functions on $E$. Given such a jet, we define the \emph{$m$'th order Taylor polynomial} $T_a^m F$ of $F$ at a point $a \in E$ by
\begin{equation}
\label{e-TaylorJet}
(T_a^m F)(x)= \sum_{k=0}^m \frac{F^{k}(a)}{k!}(x-a)^k \quad \text{for all } x \in \R.
\end{equation}
\end{definition}

If $f:\R \to \R$ is $m$-times differentiable at $a$,
the Taylor polynomial $T_a^m f$
is defined as usual using the collection $F=(D^k f)_{k=0}^m$ in \eqref{e-TaylorJet}.
We will often drop the exponent and write $T_a F$ or $T_a f$ when the order of the polynomial is obvious from the context.

\begin{definition}\label{whitneyfield}
Given a jet $F$ of order $m$ on a compact set $K\subset \R$, define for $0\leq k\leq m$ the following \emph{remainder terms}
\[(R_{a}^{m}F)^{k}(x)=F^{k}(x)-\sum_{\ell=0}^{m-k}\frac{F^{k+\ell}(a)}{\ell!}(x-a)^{\ell}
\quad
\text{for all } a, x \in K.\]
Such a jet is called a \emph{Whitney field of class $C^m$ on $K$} if for every $0\leq k\leq m$, we have
\[(R_{a}^{m}F)^{k}(b) \text{ is uniformly } o(|a-b|^{m-k}) \text{ on } K.\]
\end{definition}



The following theorem is the classical Whitney extension theorem \cite{Whitney}.

\begin{theorem}[Classical Whitney extension theorem]\label{classicalWhitney}
Let $K$ be a compact subset of an open set $U\subset \mathbb{R}$. Then for every Whitney field $F$ of class $C^{m}$ on $K$, there exists a function $f\in C^{m}(U)$ such that
\[D^kf(x)=F^{k}(x) \quad \mbox{for $0\leq k\leq m$ and $x\in K$},\]
and $f$ is $C^{\infty}$ on $U\setminus K$.
\end{theorem}

\subsubsection{$C^{m,\omega}$ Mappings}

\begin{definition}\label{defcmw}
Given a modulus of continuity $\omega$, we define $C^{m,\omega}(\R)$ to be the space of $C^{m}$ functions $f\colon \R \to \R$ such that the following seminorm is finite:
\[ \Vert f \Vert_{C^{m,\omega}(\R)} = \sup_{\substack{x,y \in \R \\ x \neq y}} 
\frac{\left| D^m f(x) - D^m f(y) \right|}{\omega(|x-y|)}.\]
Given measurable $E \subset \mathbb{R}$, we define the {\em trace space} $C^{m,\omega}(E) = \{ F|_E \, : \, F \in C^{m,\omega}(\mathbb{R}) \}$
and the seminorm
$$
\Vert f \Vert_{C^{m,\omega}(E)} = \inf \left\{  \Vert F \Vert_{C^{m,\omega}(\R)} \, : \, F \in C^{m,\omega}(\R^n), \, F|_E = f \right\}.
$$
\end{definition}

For our future study of curves in the Heisenberg group, we also define $C^{m,\omega}(E,\mathbb{R}^3)$ to be the space of functions $\gamma=(f,g,h):E \to \R^3$ such that $f,g,h \in C^{m,\omega}(E)$, and we endow this space with the seminorm
$$
\Vert \gamma \Vert_{C^{m,\omega}(E,\mathbb{R}^3)} = \Vert f \Vert_{C^{m,\omega}(E)} + \Vert g \Vert_{C^{m,\omega}(E)} + \Vert h \Vert_{C^{m,\omega}(E)}.
$$

Next we record several useful estimates which follow from \cite[(2)]{FollandRemainder}.

\begin{lemma}
For any $f \in C^{m,\omega}(\R)$ there exists a constant $C>0$ such that for any $a,b \in \R$ and $0 \leq k \leq m$
\begin{align}
\label{e-Taylor-int}
\frac{|D^k f(b) - T_a^{m-k}(D^k f)(b)|}{|b-a|^{m-k}}
\leq C \omega(|b-a|).
\end{align}
In particular, since $(T_x^mf)' = T_x^{m-1}(f')$, for all $x,y \in \R$ it follows that
\begin{equation}
\label{e-taylor1}
|f(y)-T_x^mf(y)| \leq C\omega(|x-y|) |x-y|^m,
\end{equation}
\begin{equation}
\label{e-taylor2}
|f'(y)-(T_x^mf)'(y)| \leq C\omega(|x-y|) |x-y|^{m-1}.
\end{equation}
\end{lemma}

\begin{definition}
For a compact set $K \subseteq \R$, a collection $F=(F^k)_{k=0}^m$ of continuous, real-valued functions on $K$ is a {\em Whitney field of class $C^{m,\omega}$ on $K$}
if there is a constant $C >0$ such that
$$
\frac{\left|F^k(b) - T_a^{m-k} F^k(b) \right|}{|b-a|^{m-k}} \leq C \omega(|b-a|)
    \qquad
    \text{for } a,b \in K, \, 0 \leq k \leq m.
$$
Here $T_a^{m-k} F^k$ is the $(m-k)$th order Taylor polynomial of the collection $(F^{j})_{j=k}^{m}$.

\end{definition}

A proof similar to Whitney's classical extension theorem \cite{Whitney} implies the following. (See, for example, the proof in \cite{Bierstone}.)

\begin{theorem}
\label{t-WhitClassicLip}
Suppose $F=(F^k)_{k=0}^m$ is a collection of continuous, real-valued functions on a compact set $K \subset \R$.
There is a function
$f \in C^{m,\omega}(\R)$ satisfying $D^k f|_K = F^k$ for $0 \leq k \leq m$
if and only if
$F$ is a Whitney field of class $C^{m,\omega}$ on $K$.
\end{theorem}

Whitney applied this to prove the following \cite{Whitney2}. As mentioned in the introduction, the reformulation of Whitney's ``Finiteness Principle'' given here was discovered by Brudnyi and Shvartsman \cite{BruShv,BruShv4}.

\begin{theorem}
\label{t-brush}
Suppose $K \subseteq \R$ is compact and $f:K \to \R$ is continuous.
There is some $F \in C^{m,\omega}(\R)$ with $F|_K = f$
if and only if there is a constant $M>0$ such that,
for every subset $X \subseteq K$ with $\# X = m+2$,
there is some $F_X \in C^{m,\omega}(\R)$ such that 
$F_X = f$ on $X$ and $\Vert F_X \Vert_{C^{m,\omega}(\R)} \leq M$.
\end{theorem}
\newpage

\subsubsection{$C^{\infty}$ Mappings}

\begin{definition}
A jet of order $\infty$ on a set $K\subset \R$ is a collection $F=(F^k)_{k=0}^{\infty}$ of continuous, real-valued functions on $K$. Assuming $K$ is compact, such a jet $F$ is a Whitney field of class $C^{\infty}$ on $K$ if $(F^k)_{k=0}^{m}$ is a Whitney field of class $C^{m}$ on $K$ for every $m$.
\end{definition}

The following is due to Whitney \cite{Whitney}.

\begin{theorem}[Classical Whitney extension theorem]\label{classicalWhitneyInf}
Let $K$ be a compact subset of an open set $U\subset \mathbb{R}$. Then for every Whitney field $F$ of class $C^{\infty}$ on $K$, there exists a function $f\in C^{\infty}(U)$ such that
\[D^k f(x)=F^{k}(x) \quad \mbox{for every $k\geq 0$ and $x\in K$}.\]
\end{theorem}

\subsection{Whitney Extension in The Heisenberg Group}

We now recall the (first) Heisenberg group and important quantities for Whitney extension in that setting.

\subsubsection{The Heisenberg Group} 
\label{s-Heis}

\begin{definition}\label{def-Heis}
The {\em  Heisenberg group} is defined to be $\mathbb{H} = \mathbb{R}^{3}$ with group law
\[(x,y,z)*(x',y',z')
=
\left(x+x',y+y',z+z'+2(yx'-xy')\right)\]
for $x,y,z,x',y',z' \in \R$.
With this group law, the Heisenberg group forms a Lie group.
The corresponding left invariant vector fields $X, Y, Z$ on $\mathbb{H}$ are
$$
X(p) = \tfrac{\partial}{\partial x} + 2y \tfrac{\partial}{\partial z}, \quad
Y(p) = \tfrac{\partial}{\partial y} - 2x \tfrac{\partial}{\partial z}, \quad
Z(p) = \tfrac{\partial}{\partial z}
\qquad
$$
for $p = (x,y,z) \in \R^3$.
\end{definition}

It is easy to check that $(x,y,z)^{-1} = (-x,-y,-z)$. Since $[X,Y] = -4Z$, the Lie group $\H$ is a Carnot group of step 2 with horizontal distribution $\text{span} \{ X, Y\}$.


\begin{definition}
An absolutely continuous curve $\gamma:\R \to \R^{3}$ is {\em horizontal in $\mathbb{H}$} if we have $\gamma'(t) \in \text{span} \{ X(t), Y(t)\}$ for almost every $t \in \R$. 
\end{definition}

The following proposition gives an equivalent formulation.

\begin{proposition}
Suppose $\gamma=(f,g,h):\R \to \R^3$ is absolutely continuous.
Then
$\g$ is horizontal if and only if
$$
h' = 2 (f'g - fg')
\qquad
\text{a.e. in } \R.
$$
\end{proposition}

For a proof of this, see Lemma~2.3 in \cite{Speight}. If $\g \in C^m_\H (\R)$
(i.e. $\gamma \in C^m(\R,\R^3)$ and $\g$ is a horizontal curve),
then, according to the Leibniz rule, we have
\begin{equation}
    \label{e-LeibnizRule}
D^k h = 2 \sum_{i=0}^{k-1} \binom{k-1}{i} \left(D^{k-i} f D^i g - D^{k-i} g D^i f \right)
\qquad \text{for } 1 \leq k \leq m \text{ on } \R.
\end{equation}
For a more thorough introduction to the Heisenberg group, see \cite{HajZimGeod}.

\subsubsection{Continuous Area Discrepancy and Velocity}
\label{ctsdiscrepency}

\begin{definition}
Suppose $F = (F^{k})_{k=0}^m$, $G = (G^{k})_{k=0}^m$, $H = (H^{k})_{k=0}^m$ are collections of continuous, real-valued functions on a set $E \subseteq \R$.
Set $\gamma:=(F,G,H)$. For each $a,b \in E$, define the {\em area discrepancy} $A^m(\gamma;a,b)$ by: 
\begin{align*}
A^m(\g;a,b) &= h(b) - h(a) - 2 \int_a^b ((T_a F)'T_a G - (T_a G)'T_aF)\\
& \hspace{1in} +2f(a)(g(b) - T_a G(b)) - 2g(a)(f(b) - T_aF(b)).
\end{align*}
For $a,b\in E$ with $a<b$, define the {\em $\omega$-velocity} $V^m_\omega(\g ;a,b)$ by:
$$
V^m_\omega (\g;a,b) = \omega(b-a)^2 (b-a)^{2m} + \omega(b-a) (b-a)^m \int_a^b \left( |(T_aF)'|+ |(T_aG)'| \right).
$$
If $\gamma=(f,g,h) \in C^m(\R,\R^3)$ then $A^m(\g;a,b)$ and $V^m_\omega (\g;a,b)$ are defined as above using the collections $F = (D^k f)_{k=0}^m$, $G = (D^k g)_{k=0}^m$, $H = (D^k h)_{k=0}^m$ unless otherwise noted.
If $m$ is clear from the context, we will use the notation $A(\gamma;a,b)$ and $V_{\omega}(\gamma;a,b)$ for simplicity.
\end{definition}

We recall that the area and velocity terms are left invariant. The proof of this fact is the same as that of \cite[Lemma 3.3]{ZimFinite}.

\begin{lemma}
\label{l-AVleftinv}
Suppose $\g \in C^m(\R,\R^3)$.
For any $p \in \H$ and $a,b \in \R$,
we have 
$$
A(p * \g;a,b) = A(\g;a,b) \quad \text{and} \quad V_{\omega}(p * \g;a,b) = V_{\omega}(\g;a,b).
$$
\end{lemma}

Here, $p * \gamma$ is the curve $t \mapsto p * \gamma(t)$ i.e. the left translation of $\gamma$ by $p$.
For $C^{m,\omega}$ horizontal curves in the Heisenberg group, we have the following version of Theorem~\ref{t-WhitClassicLip} from Euclidean spaces for curves in the Heisenberg group \cite{ZimSpeWhitney}.

\begin{theorem}
\label{t-HeisWhitLip}
Suppose $K \subseteq \R$ is compact and 
$F=(F^k)_{k=0}^m$, $G=(G^k)_{k=0}^m$, and $H=(H^k)_{k=0}^m$ are collections of continuous, real-valued functions on $K$.
There is a horizontal curve $\Gamma \in C^{m,\omega}(\R,\R^3)$  
satisfying $D^k \Gamma|_K = (F^k,G^k,H^k)$
for $0 \leq k \leq m$
if and only if
there is a constant $\hat{C}>0$ satisfying the following:
\begin{enumerate}
    \item 
    $F$, $G$, and $H$ are Whitney fields of class $C^{m,\omega}$ on $K$,
    \item for every $1 \leq k \leq m$, the following holds on $K$:
    $$
        H^k = 2 \sum_{i=0}^{k-1}  \binom{k-1}{i}  \left(F^{k-i}G^i- G^{k-i}F^i \right),
    $$
    \item 
    and, writing $\gamma = (F,G,H)$,
    $$
\left| A^m(\g;a,b) \right| \leq \hat{C} V^m_\omega (\g;a,b)
\quad \text{for all } a,b \in K \text{ with } a<b.
$$
\end{enumerate}
\end{theorem}

Condition {\em (1)} here was discussed above, and condition {\em (2)} is a consequence of the Leibniz rule as in \eqref{e-LeibnizRule}. Condition {\em (3)} establishes a control on the rate at which the curve gathers area in the plane, and this area is fundamentally tied to the height of a horizontal curve in the Heisenberg group. See \cite{ZimSpePinWhitney, Speight, ZimWhitney} for more discussion on this relationship.

Next we recall an interesting observation about $C^{m,\omega}$ curves in the Heisenberg group. 
If $|A/V|$ is bounded on a compact set $K \subset \mathbb{R}$, then we may redefine the third coordinate on $\mathbb{R} \setminus K$ in such a way that the resulting curve is $C^{m,\omega}$ and satisfies the horizontality condition on $K$. See Lemmas~3.5 and 3.6 in \cite{ZimFinite} for a proof.

\begin{lemma}
\label{l-horizLip}
Suppose $f, g, h \in C^{m,\omega}(\R)$
and $K \subseteq \R$ is compact. 
Write
$\gamma =(f,g,h)$.
If there is a constant $C_{AV}>0$ such that
$$
|A^m(\g;a,b)| \leq C_{AV} V_{\omega}^m (\g;a,b)
\quad
\text{for all } a,b \in K
\text{ with } a< b,
$$
then there is some $\hat{h} \in C^{m,\omega}(\mathbb{R})$
such that $\hat{h}|_K = h|_K$, 
and
$$
    \hat{h}' =  2(f'g-fg') \quad \text{ on } K.
$$
\end{lemma}

\subsubsection{Discrete Area Discrepancy and Velocity}
\label{discretediscrepency}

\begin{definition}
Let $A \subseteq \R$ and $f:A \to \mathbb{R}$. For any collection of $m+1$ distinct points $X = \{ x_0,\dots,x_m \}$ from $A$, we define the {\em $m$'th divided difference of $f$}, denoted by $f[X] = f[x_0,\dots,x_m]$, inductively as in the following paragraph.

First we define $f[x_0] = f(x_0)$ for any $x_{0}\in A$. Then, once $f[x_0,\dots,x_{k-1}]$ has been defined for any $k$ distinct points from $A$, we define $f[x_0,\dots,x_k]$ for any $k+1$ distinct points from $A$ by
\[f[x_0,\dots,x_k] = \frac{f[x_1,\dots, x_k] - f[x_0,\dots,x_{k-1}]}{x_k-x_0}.\]

For a map $\gamma:A \to \R^3$, we define componentwise $\gamma[X] := (f[X],g[X],h[X])$.
\end{definition}

\begin{definition}
Given $A \subseteq \R$, $f:A \to \R$, and a
finite set $X= \{x_0,\dots,x_k\} \subseteq A$, 
the associated Newton interpolation polynomial is defined as
\begin{align*}
P(X;f)(x)
= f[x_0]+(x-x_0)f[x_0,x_1] 
+ \cdots+ 
(x-x_0)\cdots(x-x_{k-1})f[x_0,\dots,x_k].
\end{align*}
This is the unique polynomial of degree at most $k$ satisfying
$P(x_i) = f(x_i)$ for $i= 0,\dots,k$.
\end{definition}

Such a polynomial is an excellent approximation of $f$ at each of the interpolated points $\{x_0,\dots,x_k\}$. We will state this quantitatively here.
For a proof of this, see \cite[Lemma~2.8]{ZimFinite}.

\begin{lemma}
\label{l-poly}
Suppose 
$\omega$ is a modulus of continuity and
$f \in C^{m,\omega}(I)$ for some compact interval $I$.
There is a constant $C>0$ 
depending only on $m$ and $\Vert f \Vert_{C^{m,\omega}(I)}$
such that, for any
$X \subseteq I$
with $\#X = m+1$ 
and $P = P(X;f)$,
\begin{equation*}
\label{e-poly2}   
\frac{|f(x)-P(x)|}{\diam(X)^{m}} \leq C\omega (\diam(X))
\quad
\text{and}
\quad
\frac{|f'(x)-P'(x)|}{\diam(X)^{m-1}} \leq C\omega (\diam(X))
\end{equation*}
for all $x \in [\min X,\max X]$.
\end{lemma}

We now use interpolation polynomials to define new discrete versions of the area discrepancy and velocity. The advantage of these discrete versions is the absence of any derivatives of $f$ and $g$ from their definitions. This allows us to discuss the ``horizontality'' of a continuous curve without knowing anything about its derivatives.

\begin{definition}
Fix $E \subseteq \R$ and $\gamma=(f,g,h):E \to \mathbb{R}^3$.
Suppose $X \subseteq E$ with $\# X = m+1$, 
and 
set
$P_f = P(X;f)$ and $P_g = P(X;g)$.
For any $a,b \in X$, define the {\em discrete area discrepancy} $A[X,\g;a,b]$ 
and {\em discrete $\omega$-velocity} $V_{\omega} [X,\g;a,b]$ as follows:
\begin{align*}
A[X,\g;a,b] &= h(b) - h(a) - 2 \int_{a}^{b} (P_f'P_g - P_g'P_f),\\
V_{\omega}[X,\g;a,b] &= \omega(\diam X)^2 (\diam X)^{2m} + \omega(\diam X) (\diam X)^m \int_a^b \left(|P_f'|+ |P_g'|\right).
\end{align*}
\end{definition}



Again, the area and velocity terms are left invariant \cite[Lemma 4.3]{ZimFinite}.

\begin{lemma}
\label{l-AVleftinvDisc}
Suppose $\g :X \to \H$ for some
$X \subset \R$ with $\# X = m+1$.
Fix $a,b \in X$ and $p \in \H$.
Then 
$$
A[X,p *\g;a,b] = A[X,\g;a,b]
\quad
\text{and}
\quad
V_{\omega}[X,p * \g;a,b] = V_{\omega}[X,\g;a,b].
$$
\end{lemma}

\subsection{Lusin Approximation}

Finally, we introduce the Lusin approximation properties that we will consider and the conditions that will be used to obtain them.

\begin{definition}\label{Lusinm}
A horizontal curve $\Gamma\colon [a,b]\to \mathbb{H}$ has the \emph{Lusin property of order $m$ and modulus of continuity $\omega$} if, for every $\epsilon>0$, there exists a $C^{m,\omega}$ horizontal curve $\widetilde{\Gamma}\colon [a,b]\to \mathbb{H}$ such that
\[\mathcal{L}^{1}(\{x\in [a,b]: \widetilde{\Gamma}(x)\neq \Gamma(x)\})<\varepsilon.\]
We also refer to this as the $(m,\omega)$-Lusin property.

A horizontal curve $\Gamma\colon [a,b]\to \mathbb{H}$ has the \emph{Lusin property of order $\infty$} if, for every $\epsilon>0$, there exists a $C^{\infty}$ horizontal curve $\widetilde{\Gamma}\colon [a,b]\to \mathbb{H}$ such that
\[\mathcal{L}^{1}(\{x\in [a,b]: \widetilde{\Gamma}(x)\neq \Gamma(x)\})<\varepsilon.\]
We also refer to this as the $\infty$-Lusin property.
\end{definition}

The following $L^{1}$ differentiability condition was used in \cite{CapPinSpe} to prove a Lusin result for horizontal curves of class $C^{m}$.

\begin{definition}\label{L1}
A measurable map $u:[a,b]\to \mathbb{R}$ is \emph{$m$-times $L^{1}$ differentiable at a point $x\in (a,b)$} if there exists a polynomial $P_{u,x}^m\colon \mathbb{R}\to \mathbb{R}$ of degree at most $m$ such that:
     \[
\dashint_{B(x,\rho)} |u(y)-P_{u,x}^m(y)|\, dy = o(\rho^m).
\]
\end{definition}

Motivated by this, we introduce the following definition.

\begin{definition}\label{L1w}
A measurable map $u:[a,b]\to \mathbb{R}$ is \emph{$m$-times $L^{1,\omega}$ differentiable at $x\in (a,b)$} if there exists a polynomial $P_{u,x}^m\colon \mathbb{R}\to \mathbb{R}$ of degree at most $m$ and constants $C>0$ and $\rho_0>0$ with $B(x,\rho_{0})\subset (a,b)$ such that 
\[
\dashint_{B(x,\rho)} |u(y)-P_{u,x}^m(y)|\, dy\leq C\omega(\rho)\rho^m
\text{ for all } 0<\rho<\rho_0.
\]
\end{definition}

Note that Definition \ref{L1} and Definition \ref{L1w} also make sense for measurable maps $u$ which are defined almost everywhere on $[a,b]$. 

\begin{remark}\label{diffimp}
Suppose a function $u$ is $m$-times $L^{1,\omega}$ differentiable at a point $x$. Then $u$ is also $m$-times $L^{1}$ differentiable and $m$-times approximately differentiable at $x$ with the same choice of polynomial $P_{u,x}^m$ \cite{CapPinSpe}. This follows because $\omega\colon [0,\infty)\to [0,\infty)$ is continuous and $\omega(0)=0$.
\end{remark}

\section{A $C^{m,\omega}$ Finiteness Principle}
\label{s-finiteness}

In this section we prove our first main result (Theorem \ref{t-finiteness}), a $C^{m,\omega}$ version of the finiteness principle from \cite{ZimFinite}. This new result avoids the drop in regularity found therein. 

We first show that the boundedness of the ratios $A/V$ are equivalent, regardless of whether the continuous or discrete area and velocity are used.

\begin{lemma}
\label{l-AVLip2}
Suppose $\gamma \in C^{m,\omega}(\R,\mathbb{R}^3)$, and suppose $K \subseteq \R$ is compact with finitely many isolated points and $\#K \geq m+1$.
Then
there is a constant $C_{AV}>0$ such that
\begin{equation}
\label{e-AVcont}
|A(\g;a,b)| \leq C_{AV} V_{\omega} (\g;a,b)
\quad
\text{for all } a,b \in K
\text{ with } a< b
\end{equation}
if and only if
there is a constant $C_{dAV}>0$ such that,
for any $X \subseteq K$ with $\# X = m+1$,
\begin{equation}
\label{e-AVdisc}
A[X,\g ; a,b] \leq C_{dAV} V_{\omega}[X,\g ; a,b]
\quad
\text{for all } a,b \in X \text{ with } a<b.
\end{equation}

Moreover, suppose $I$ is a compact interval containing $K$.
If \eqref{e-AVcont} holds, then $C_{dAV}$ can be chosen to depend only on $m$, $\Vert f \Vert_{C^{m,\omega}(I)}$, $\Vert g \Vert_{C^{m,\omega}(I)}$, and $C_{AV}$.
If \eqref{e-AVdisc} holds for any $X \subseteq K$ with $\# X = m+1$, then $C_{AV}$ can be chosen to depend only on $m$, $\Vert f \Vert_{C^{m,\omega}(I)}$, $\Vert g \Vert_{C^{m,\omega}(I)}$, and $C_{dAV}$.
\end{lemma}

\begin{proof}
Assume first that 
\eqref{e-AVdisc} holds for any $X \subseteq K$ with $\# X = m+1$.
We may clearly assume that $K$ is not a finite set.
Let $a,b \in K$ with $a<b$.
By our assumption on $K$, 
we may assume that either $a$ or $b$ is a limit point of $K$.
Thus we can choose a finite set $X$ consisting of $a$, $b$, and $m-1$ other distinct points in $K$ within $(b-a)/2$ of $a$ or $b$.
In particular, $\diam X \leq 2(b-a)$.

By Lemmas~\ref{l-AVleftinv} and \ref{l-AVleftinvDisc},
we may assume without loss of generality that $\g(a) = 0$.
For simplicity, 
write $A = A(\g;a,b)$ and $V=V_{\omega}(\g;a,b)$,
and write
$A_X = A[X,\g;a,b]$ and $V_X = V_{\omega}[X,\g;a,b]$.
Note that
\begin{align}
\left| \frac{A}{V} \right |
\leq \left| \frac{A}{V} - \frac{A_X}{V_X}\right| + \left| \frac{A_X}{V_X}\right|
\leq
\left| \frac{A_X}{V_X} \right|
\left| \frac{V - V_X}{V}\right|
+
\left| \frac{A - A_X}{V}\right|
+
\left| \frac{A_X}{V_X}\right|.
\label{e-AV2}
\end{align}
Notice first that $|A_X/V_X| \leq C_{dAV}$.
We will now provide bounds for the remaining terms. 
The proof of this is similar to that of Lemma~4.4 in \cite{ZimFinite}.
As such, we will leave out some of the details.
We begin by writing
\begin{align}
\label{e-AminusA}
A - A_X 
= 
2 \int_a^b \left[ (P_f'P_g  - (T_af)'T_ag) + ((T_ag)'T_af - P_g'P_f) \right]
\end{align}
and
\begin{align*}
P_f'P_g  - (T_af)'T_ag
=
(T_af)'(P_g - T_ag)
&+T_ag(P_f'  - (T_af)')
\\
&\quad +(P_f' - (T_af)')(P_g-T_ag).
\end{align*}
Suppose $I$ is a compact interval containing $K$.
Properties \eqref{e-taylor1} and \eqref{e-taylor2} and Lemma~\ref{l-poly} give
a constant $C>0$ depending only on $m$, $\Vert f \Vert_{C^{m,\omega}(I)}$, $\Vert g \Vert_{C^{m,\omega}(I)}$, and $C_{dAV}$ such that
\begin{align}
\label{e-1st}
|P_g  - T_ag|
\leq
|P_g - g| + |g - T_ag|
\leq
C \omega(\diam X) (\diam X)^m
\end{align}
and
\begin{align}
\label{e-2nd}
|P_f' - (T_af)'|
\leq
|P_f' - f'| + |f' - (T_af)'|
\leq
C \omega(\diam X) (\diam X)^{m-1}
\end{align}
on $[a,b]$.
Therefore, 
\begin{align}
\label{e-3rd}
|(T_af)' - P_f'||P_g-T_ag| \leq C^2 \omega(\diam X)^2 (\diam X)^{2m-1}
\end{align}
on $[a,b]$.
Equations \eqref{e-1st}, \eqref{e-2nd}, and \eqref{e-3rd} together 
with the bound $\diam X \leq 2(b-a)$ and the fact that $\omega$ is increasing and subadditive
imply
\begin{align*}
    \int_a^b |P_f'P_g  - (T_af)'T_ag|
    &\lesssim_C 
    \omega(b-a)^2 (b-a)^{2m} \\
    &\qquad +
    \omega(b-a) (b-a)^m \int_a^b |(T_af)'| +
    \omega(b-a) (b-a)^{m-1} \int_a^b |T_ag| \\
    &\lesssim_{C,m}
     \omega(b-a)^2 (b-a)^{2m}
     +
     \omega(b-a) (b-a)^m \int_a^b |(T_af)'| + |(T_ag)'|.
\end{align*}
In the last line, we
applied Corollary~2.11 in \cite{ZimSpePinWhitney}
to the polynomial $T_ag$ as in the proof of \cite[Lemma~4.4]{ZimFinite}.
Arguing similarly for the other term in \eqref{e-AminusA}, we may conclude that 
$
|A - A_X |
\lesssim_{C,m}
V
$.

Moreover, we have
\begin{align*}
    |V - V_X| 
    &\leq 
    \omega (b-a)^2 (b-a)^{2m}
    +
    \omega(\diam X)^2 (\diam X)^{2m}\\
    &\hspace{.5in}
    +
    \left|\omega(b-a)(b-a)^m - \omega(\diam X)(\diam X)^m\right|
    \int_a^b\left(|(T_af)'| + |(T_ag)'|\right)\\
    & \hspace{.5in}
    +
    \omega(\diam X) (\diam X)^m\int_a^b\left||(T_af)'| - |P_f'|+ |(T_ag)'| - |P_g'| \right|\\
    &\lesssim_{C,m}
    \omega(b-a)^2(b-a)^{2m} + \omega(b-a)(b-a)^m
    \int_a^b\left(|(T_af)'| + |(T_ag)'|\right) = V.
\end{align*}
In the last line, we applied \eqref{e-2nd}.
These are the desired bounds for \eqref{e-AV2}, so \eqref{e-AVcont} holds.

To prove the reverse implication, assume now that \eqref{e-AVcont} holds on $K$.
Suppose $X$ is a set of $m+1$ distinct points in $K$, and choose $a,b \in X$ with $a < b$.
As above,
assume that $\g(a) = 0$, and write $A$, $V$, $A_X$, and $V_X$ as before.
We now write
\begin{align}
\left| \frac{A_X}{V_X} \right |
\leq \left| \frac{A_X}{V_X} - \frac{A}{V}\right| + \left| \frac{A}{V}\right|
\leq
\left| \frac{A - A_X}{V_X}\right|
+
\left| \frac{A}{V} \right|
\left| \frac{V - V_X}{V_X}\right|
+ \left| \frac{A}{V}\right|.
\label{e-AV1}
\end{align}
Here, $|A/V| \leq C_{AV}$.
It remains to bound the other terms.
We once again have \eqref{e-AminusA}.
This time, we write
\begin{align*}
P_f'P_g  - (T_af)'T_ag
=
P_f'(P_g - T_ag)
&+P_g(P_f'  - (T_af)')
\\
&\quad -(P_f' - (T_af)')(P_g-T_ag).
\end{align*}
Exactly as above, we have the estimates \eqref{e-1st} through \eqref{e-3rd}
for a constant $C>0$
depending only on $m$, $\Vert f \Vert_{C^{m,\omega}(I)}$, $\Vert g \Vert_{C^{m,\omega}(I)}$, and $C_{AV}$,
and we can apply Corollary~2.11 from \cite{ZimSpePinWhitney} to $P_g$ to conclude that 
\begin{align*}
    \int_a^b |P_f'P_g  - (T_af)'T_ag|
    \lesssim_{C,m}
     \omega(\diam X)^2 (\diam X)^{2m}
     +
     \omega(\diam X) (\diam X)^m \int_a^b |P_f'| + |P_g'|.
\end{align*}
A similar estimate again holds for the other term in \eqref{e-AminusA}, so we may conclude that $
|A - A_X |
\lesssim_{C,m} V_X$.
Since $b-a \leq \diam X$, we can also argue as before to conclude that 
\begin{align*}
    |V - V_X| 
    \lesssim_{C,m}
    \omega(\diam X)^2(\diam X)^{2m} + \omega(\diam X)(\diam X)^m
    \int_a^b\left(|P_f'| + |P_g'|\right) = V_X.
\end{align*}
This proves that \eqref{e-AVdisc} holds for any $X \subseteq K$ with $\# X = m+1$, and the proof is complete.
\end{proof}

We now establish our first main result.
Again, the main difference between this result and Theorem~1.7 in \cite{ZimFinite} is the fact that our extension now has $C^{m,\omega}$ regularity rather than only $C^{m,\sqrt{\omega}}$ regularity.
While the proof of this is similar to that of \cite[Theorem~1.7]{ZimFinite}, we must be more careful with our treatment of the constants.

\begin{theorem}
\label{t-finiteness}
Assume $K \subseteq \mathbb{R}$ is compact
with finitely many isolated points
and $\# K \geq m+2$
for some positive integer $m$.
Suppose $\g:K \to \H$ is continuous.
Then there is a horizontal curve $\Gamma \in C^{m,\omega}(\R,\R^3)$
with $\Gamma|_K = \g$
if and only if
there exists a constant $M > 0$ such that,
for any $X \subseteq K$ with $\#X = m+2$,
there is a curve $\Gamma_X \in C^{m,\omega}(\R,\R^3)$
with $\Gamma_X = \gamma$ on $X$,
$
\Vert \Gamma_X \Vert_{C^{m,\omega}(\R,\R^3)}  \leq M,
$
and
$$
\left| A(\Gamma_X;a,b) \right| \leq M V_{\omega}(\Gamma_X;a,b)
\quad \text{for all } a,b \in K \text{ with } a<b.
$$
\end{theorem}

\begin{proof}
Necessity is proven in Proposition~3.2 from \cite{ZimSpeWhitney}.
We need only prove sufficiency.
Suppose $K \subseteq \R$ is compact with finitely many isolated points and $\#K \geq m+2$.
Suppose $M>0$ and $\gamma:K\to \R^3$ satisfy the following:
for any $X \subseteq K$ with $\#X = m+2$,
there is a function $\Gamma_X \in C^{m,\omega}(\R,\R^3)$
such that $\Gamma_X = \gamma$ on $X$,
$
\Vert \Gamma_X \Vert_{C^{m,\omega}(\R,\R^3)} \leq M,
$
and
\begin{equation}
\label{e-assump}    
\left| A(\Gamma_X;a,b) \right| \leq M V_{\omega}(\Gamma_X;a,b)
\quad \text{ for all } a,b \in X \text{ with } a<b.
\end{equation}
As in the proof of \cite[Theorem~1.7]{ZimFinite}, it follows from 
\cite[Theorem~A]{BruShv} that 
there is some $\tilde{\g} =(f,g,h) \in C^{m,\omega}(\R,\R^3)$ such that $\tilde{\g}|_K = \g$.

We will now observe that 
$\tilde{\gamma}$ satisfies 
\eqref{e-AVdisc} for any $X \subseteq K$ with $\# X = m+1$.
Indeed, choose a set $X \subset K$ with $\# X = m+1$, and fix $a,b \in X$ with $a < b$.
Set $Y = X \cup \{ x \}$ for some $x \in K \setminus X$,
and choose $\Gamma_Y$ as in the hypothesis.
In particular, $\Gamma_Y$ satisfies \eqref{e-AVcont} on $Y$
with constant $C_{AV}=M$.
Lemma~\ref{l-AVLip2} implies that $\Gamma_Y$ satisfies 
\eqref{e-AVdisc} on $X$ as well with a constant $M'>0$ depending only on
$m$, $\Vert f \Vert_{C^{m,\omega}(I)}$, $\Vert g \Vert_{C^{m,\omega}(I)}$, and $M$.
Therefore, since $\Gamma_Y = \gamma = \tilde{\gamma}$ on $K$, we have
$$
\left| A[X,\tilde{\g};a,b] \right| = \left| A[X,\Gamma_Y;a,b] \right| \leq M'  V_{\omega} [X,\Gamma_Y;a,b]
= M'  V_{\omega} [X,\tilde{\g};a,b].
$$
Therefore, $\tilde{\gamma}$ satisfies \eqref{e-AVdisc},
so Lemma~\ref{l-AVLip2} implies that $\tilde{\gamma}$ satisfies \eqref{e-AVcont} on $K$
with a constant depending only on $m$, $\Vert f \Vert_{C^{m,\omega}(I)}$, $\Vert g \Vert_{C^{m,\omega}(I)}$, and $M'$.
Lemma~\ref{l-horizLip} then allows us to choose
some
$\hat{h} \in C^{m,\omega}(\mathbb{R})$
such that $\hat{h}|_K = h|_K$ and
$\hat{h}' =  2(f'g-fg')$ on $K$.

Set $\hat{\g} = (f,g,\hat{h})$.
Thus the collections
$\left( D^k \hat{\g} \right)_{k=0}^m$ on $K$
satisfy
conditions {\em (1)} and {\em (2)} from  Theorem~\ref{t-HeisWhitLip}.
Fix $a,b \in K$ with $a<b$.
Notice that $A(\hat{\g};a,b) = A(\tilde{\g};a,b)$ since $\hat{h} = h$ on $K$.
Also, $V_{\omega}(\hat{\g};a,b) = V_{\omega}(\tilde{\g};a,b)$.
Since $\tilde{\gamma}$ satisfies \eqref{e-AVcont} on $K$, it follows immediately that $\hat{\g}$ satisfies \eqref{e-AVcont} on $K$ also
i.e. $\hat{\g}$ satisfies {\em (3)} from Theorem~\ref{t-HeisWhitLip}.
We may therefore apply Theorem~\ref{t-HeisWhitLip} to find
a horizontal curve $\Gamma \in C^{m,\omega}(\R,\R^3)$ such that $\Gamma|_K = \hat{\g}|_K = \tilde{\g}|_K = \g$.
\end{proof}

\section{$C^{m,\omega}$ Lusin Approximation of Curves}
\label{s-lusin}

In this section we prove our second main result (Theorem \ref{t-Lusin}), a $C^{m,\omega}$ Lusin approximation result for horizontal curves. This is an analogue of a similar result in the $C^{m}$ setting in \cite{CapPinSpe}. We begin with some simple remarks.

\begin{remark}\label{gain}
Let $m\in \mathbb{N}$ and $\omega(t)= t$. Suppose a function $u\colon [a,b]\to \mathbb{R}$ is $m$-times $L^{1}$ differentiable at a point $x\in (a,b)$ with associated polynomial $P^m_{u,x}(y)=\sum_{i=0}^m \frac{a_i(x)}{i!}(y-x)^i$ with $a_i(x)\in \mathbb{R}$. Then $u$ is also $(m-1)$-times $L^{1,\omega}$ differentiable at $x$ with polynomial $P^{m-1}_{u,x}(y)=\sum_{i=0}^{m-1} \frac{a_i(x)}{i!}(y-x)^i$.

To see this, it follows from the definition that there exists $\rho_0>0$ with $B(x,\rho_{0})\subset (a,b)$ such that for all $0<\rho<\rho_0$, we have
    \[
\dashint_{B(x,\rho)} |u(y)-P_{u,x}^m(y)|\, dy\leq \rho^m.
    \]
 Then we have,
\begin{align*}
\dashint_{B(x,\rho)} |u(y)-P_{u,x}^{m-1}(y)|\, dy&\leq \dashint_{B(x,\rho)} |u(y)-P_{u,x}^{m}(y)|\, dy + \frac{a_m(x)}{m!}\dashint_{B(x,\rho)} |y-x|^m\, dy\\
&\leq \rho^{m} + \frac{a_m(x)}{m!}\rho^m 
\leq \left(1 + \frac{a_m(x)}{m!}\right) \omega(\rho) \rho^{m-1} 
\end{align*}
where 
we recall that $\omega(t)=t$.
\end{remark}

\begin{remark}\label{rem44}
Suppose $u\colon [a,b]\to \mathbb{R}$ is $m$-times $L^{1}$ differentiable almost everywhere in $[a,b]$ with associated polynomial $P^m_{u,x}$ described above.
Then the coefficient functions $x \mapsto a_i(x)$ are measurable for $i = 0,1,\dots,m$ almost everywhere on $[a,b]$.
Indeed, 
it follows as in Remark~2.3 from \cite{AlbBiaCri} that $u$ is $m$-times approximately differentiable with the same polynomials $P^m_{u,x}$, and the proof of the almost everywhere measurability of $a_i$ in this setting can be found on page 194 of \cite{LT94}. A similar remark holds for $L^{1,\omega}$ differentiability.
\end{remark}

\begin{remark}\label{rem3}
Let $m\in\mathbb{N}$ and $\omega(t)=t$. Suppose that a function $u:[a,b]\to\mathbb{R}$ is $m$-times $L^{1,\omega}$ differentiable at a point $x\in (a,b)$ with polynomial $P^m_{u,x}(y)=\sum_{i=0}^m \frac{a_i(x)}{i!}(y-x)^i$. Then $u$ is $(m-1)$-times $L^{1,\omega}$ differentiable at $x$ with polynomial $P^{m-1}_{u,x}(y)=\sum_{i=0}^{m-1} \frac{a_i(x)}{i!}(y-x)^i$. 

Indeed, it follows from the definitions that $u$ is $m$-times $L^{1}$ differentiable at $x$ with the same polynomial $P^m_{u,x}(y)$. It then follows from Remark \ref{gain} that $u$ is $(m-1)$-times $L^{1,\omega}$ differentiable at $x$ with polynomial $P^{m-1}_{u,x}(y)$.

\end{remark}

The proof of the following lemmas are almost the same as that of \cite[Lemma 3.1]{CapPinSpe} and \cite[Lemma 3.2]{CapPinSpe}. We provide the proofs for completeness.
Recall that $\omega$ is once again an arbitrary modulus of continuity.

\begin{lemma}\label{intL1}
Let $f\colon [a,b]\to \mathbb{R}$ be absolutely continuous and $m\geq 2$.

Suppose $f'$ is $m-1$ times $L^{1,\omega}$ differentiable at a point $x\in (a,b)$ with $L^{1,\omega}$ derivative 
$P_{f,x}^{m-1}$ of degree at most $m-1$. Then $f$ is $m$ times $L^{1,\omega}$ differentiable at $x$ with $L^{1,\omega}$ derivative $Q_{f,x}^{m}$ of degree at most $m$ defined by \[Q_{f,x}^m(y):=f(x)+\int_{x}^{y}P_{f,x}^{m-1}(t)\, dt.\]
\end{lemma}

\begin{proof}
Denote $P=P_{f,x}^{m-1}$ and $Q=Q_{f,x}^m$. From the definition of $L^{1,\omega}$ differentiability, there is $C>0$ and $\rho_0>0$ with $B(x,\rho_{0})\subset (a,b)$ so that for all $0<\rho<\rho_0$,
\begin{equation*}
\dashint_{B(x,\rho)} |f'(t) - P(t)|\, dt \leq C \omega(\rho) \rho^{m-1}.
\end{equation*}
Fixing such a ball $B(x,\rho)$, absolute continuity of $f$ gives for all $y\in B(x,\rho)$
\begin{equation*}
\begin{split}
|f(y) - Q(y)| &= \left| \left(f(x) + \int_x^y f'(t)\, dt\right) - \left( f(x) + \int_x^y P(t)\, dt \right) \right| \\
& = \left| \int_x^y (f'(t) - P(t))\, dt \right| \\
&\leq 2 \rho \, \dashint_{B(x,\rho)} |f'(t) - P(t)|\, dt \\
& \leq C \omega(\rho) \rho^{m}.
\end{split}
\end{equation*}
Hence,
\[\dashint_{B(x,\rho)} |f(y)-Q(y)|\, dy \leq C\omega(\rho)\rho^m.\]
This completes the proof.
\end{proof}

\begin{lemma}\label{vertL1}
Suppose $(f,g,h)\colon [a,b]\to \mathbb{H}$ is a horizontal curve in $\mathbb{H}$ and $f', g'$ are $m-1$ times $L^{1,\omega}$ differentiable at a point $x\in (a,b)$ for some $m\geq 2$. Then $h'$ is $m-1$ times $L^{1,\omega}$ differentiable at $x$. More precisely, denote
\[R:=2(P'Q-Q'P),\]
where $P, Q$ are the $L^{1,\omega}$ derivatives of order $m$ of $f, g$ respectively which exist by Lemma~\ref{intL1}. Let $\widetilde{R}$ be the polynomial of degree at most $m-1$ such that $R(y)-\widetilde{R}(y)$ is divisible by $(y-x)^{m}$. Then $\widetilde{R}$ is the $L^{1,\omega}$ derivative of $h'$ of order $m-1$ at $x$.

\end{lemma}

\begin{proof}
Fix $C>0$ and $\delta>0$ with $B(x,\delta)\subset (a,b)$ such that for all $0<\rho<\delta$ we have
\[\dashint_{B(x,\rho)}|f-P|\leq C\omega(\rho)\rho^{m}, \qquad \dashint_{B(x,\rho)}|f'-P'|\leq C\omega(\rho)\rho^{m-1},\]
and
\[\dashint_{B(x,\rho)}|g-Q|\leq C\omega(\rho) \rho^{m}, \qquad \dashint_{B(x,\rho)}|g'-Q'|\leq C\omega(\rho) \rho^{m-1}.\]
Let $0<\rho<\min( \delta , 1)$. Since $(f,g,h)$ is a horizontal curve,
\begin{align*}
\dashint_{B(x,\rho)} |h'-R| &= \dashint_{B(x,\rho)} |2(f'g-g'f)-2(P'Q-Q'P)|\\
&\leq 2 \dashint_{B(x,\rho)}|f'g-P'Q|+2 \dashint_{B(x,\rho)}|Q'P-g'f|.
\end{align*}
We show how to estimate the first term as the bound for the second is similar. Notice \[f'g-P'Q=(f'-P')g+P'(g-Q).\] Since $g$ and $P'$ are continuous and hence bounded on $[a,b]$, we can continue our estimate as follows:
\begin{align*}
\dashint_{B(x,\rho)} |f'g-P'Q|&\leq \|g\|_{\infty} \dashint_{B(x,\rho)}|f'-P'| + \|P'\|_{\infty}\dashint_{B(x,\rho)}|g-Q| \\
&\leq C\|g\|_{\infty} \omega(\rho)\rho^{m-1} + C\|P'\|_{\infty}\omega(\rho)\rho^{m}\\
&\leq \tilde C \omega(\rho) \rho^{m-1}
\end{align*}
for a constant $\tilde C$ independent of $\rho$. The estimate of $\dashint_{B(x,\rho)}|Q'P-g'f|$ is similar. Hence $\dashint_{B(x,\rho)} |h'-R|\leq  \tilde C \omega(\rho) \rho^{m-1}$. Next, after possibly increasing $\tilde{C}$, we have
\begin{align*}
\dashint_{B(x,\rho)}|h'-\widetilde{R}| &\leq \dashint_{B(x,\rho)} |h'-R|+\dashint_{B(x,\rho)}|R-\widetilde{R}|\\
&\leq \tilde C \omega(\rho) \rho^{m-1}+\tilde C\rho^{m}.
\end{align*}
The conclusion follows because 
$t\mapsto \omega(t)/t$ is decreasing on $(0,\infty)$ \cite[Lemma 2.5]{ZimSpeWhitney},
and hence
$t\leq \hat{C}\omega(t)$ for all $t\in (0,1)$ where $\hat{C}\geq 1$ is a constant depending only on $\omega$.
\end{proof}

The following lemma shows the parameters $C$ and $\rho_0$ in Definition~\ref{L1w} can be made independent of the point $x$ by discarding a set of small measure.
This was not necessary in the proof of \cite[Theorem 4.1]{CapPinSpe} as the convergence was not controlled by a modulus of continuity. The introduction of $\omega$ makes these uniform bounds essential.

\begin{lemma}\label{uniformparameters}
Suppose $u \in L^{1}([a,b])$ is $m$-times $L^{1,\omega}$ differentiable at $x$ for almost every $x\in (a,b)$. Then for every $\varepsilon>0$, there exists a compact set $K\subset [a,b]$ with $\mathcal{L}^1 ([a,b]\setminus K)<\varepsilon$ so that $u$ is $m$-times $L^{1,\omega}$ differentiable at every $x\in K$ with uniform parameters.
In other words, there exist $C>0$ and $\rho_{0}>0$ such that, for every $x\in K$, the following is true: $B(x,\rho_{0})\subset (a,b)$ and there exist $a_{i}(x)\in \mathbb{R}$ for $0\leq i\leq m$ such that for all $0<\rho<\rho_{0}$,
\[
\dashint_{B(x,\rho)} |u(y)- \sum_{i=0}^{m}a_{i}(x)(y-x)^{i}|\, dy\leq C\omega(\rho)\rho^m.
\]
\end{lemma}

\begin{proof}
By hypothesis, for almost every $x\in (a,b)$ there exist $a_{i}(x)\in \mathbb{R}$ for $0\leq i\leq m$
and constants $C_x>0$, $\rho_{x}>0$ with $B(x,\rho_{x})\subset (a,b)$ such that, for all $0<\rho<\rho_{x}$,
\[
\dashint_{B(x,\rho)} |u(y)- \sum_{i=0}^{m}a_{i}(x)(y-x)^{i}|\, dy\leq C_x \omega(\rho)\rho^m.
\]
For each $N\in \mathbb{N}$, let $A_{N}$ be the set of $x\in (a,b)$ such that $B(x,1/N)\subset (a,b)$ and, for every $0<\rho<1/N$,
\[
\dashint_{B(x,\rho)} |u(y)- \sum_{i=0}^{m}a_{i}(x)(y-x)^{i}|\, dy\leq N\omega(\rho)\rho^m.
\]

\smallskip

\noindent \textbf{Claim:} To prove the lemma it suffices to show that $A_{N}$ is measurable for each $N\in \mathbb{N}$.

\begin{proof}[Proof of Claim] Suppose $A_{N}$ is measurable for every $N\in \mathbb{N}$. By the hypothesis of the lemma, we have $[a,b]\setminus Z=\bigcup_{N=1}^{\infty}A_{N}$ for some measure zero set $Z\subset [a,b]$. Clearly $A_{N}\subset A_{N+1}$ for every $N$. Hence 
\[\mathcal{L}^1 ([a,b]) = \mathcal{L}^1\left(\bigcup_{N=1}^{\infty}A_{N} \right)=\lim_{N\to \infty}\mathcal{L}^1(A_{N}).\]
Given $\varepsilon>0$, it follows that $\mathcal{L}^1 ([a,b]\setminus A_{N})<\varepsilon$ for sufficiently large $N$. Using the definition of $A_{N}$, this proves the lemma under the assumption that $A_{N}$ is measurable for each $N\in \mathbb{N}$.
\end{proof}

\smallskip

\noindent \textbf{Claim:} The set $A_{N}$ is measurable for each $N\in \mathbb{N}$

\begin{proof}[Proof of Claim] Given $0<\rho<1/N$, let $A_{N,\rho}$ be the set of those points $x\in (a,b)$ so that $B(x,1/N)\subset (a,b)$ and
\[
\dashint_{B(x,\rho)} |u(y)- \sum_{i=0}^{m}a_{i}(x)(y-x)^{i}|\, dy\leq N\omega(\rho)\rho^m.
\]
We claim that 
\begin{equation}\label{intersect}
A_{N}=\bigcap_{\substack{\rho \in (0,1/N)\cap \mathbb{Q}}} A_{N,\rho}.
\end{equation}
Clearly $A_N$ is a subset of the right hand side of \eqref{intersect}. To see the opposite, suppose $x$ belongs to the right side. Fix any $\rho^{\ast}\in (0,1/N)$, possibly irrational. Choose a sequence $(\rho_{n}) \subset \mathbb{Q}\cap (0,1/N)$ with $\rho_{n}\downarrow \rho^{\ast}$. Then, using the fact $\rho^{\ast}\leq \rho_{n}$ for every $n$,
\begin{align*}
&\dashint_{B(x,\rho^{\ast})} |u(y)- \sum_{i=0}^{m}a_{i}(x)(y-x)^{i}|\, dy\\
&\qquad \leq \frac{\rho_{n}}{\rho^{\ast}} \dashint_{B(x,\rho_{n})} |u(y)- \sum_{i=0}^{m}a_{i}(x)(y-x)^{i}|\, dy\\
&\qquad \leq \frac{\rho_{n}}{\rho^{\ast}} N\omega(\rho_{n})\rho_{n}^{m}.
\end{align*}
Taking the limit as $n\to \infty$ and using the continuity of $\omega$ gives
\[ \dashint_{B(x,\rho^{\ast})} |u(y)- \sum_{i=0}^{m}a_{i}(x)(y-x)^{i}|\, dy \leq N\omega(\rho^{\ast})(\rho^{\ast})^{m}.\]
Since $\rho^{\ast}$ was arbitrary, this proves the equality \eqref{intersect}.

Using \eqref{intersect}, then, it suffices to show that $A_{N,\rho}$ is measurable for each fixed $N$ and $\rho \in (0,1/N) \cap \mathbb{Q}$. To do so, it is enough to show measurability of the function $\phi$ defined almost everywhere on $[a+1/N,b-1/N]$ by
\[\phi(x):= \int_{B(x,\rho)} |u(y)- \sum_{i=0}^{m}a_{i}(x)(y-x)^{i}|\, dy.\]

To this end, recall from Remark~\ref{rem44} that the coefficient functions $x\mapsto a_{i}(x)$ are almost everywhere defined and measurable on $[a,b]$ for $0\leq i\leq m$. By Lusin's theorem, then, it suffices to show that $\phi$ is measurable when restricted to any compact subset $K\subset [a+1/N,b-1/N]$ with the following property:  for all $0\leq i\leq m$, $a_{i}$ exists at every point of $K$ and defines a continuous function on $K$. 

Fix such a compact set $K$. We will show that $\phi$ is continuous on $K$. Indeed, for $x,z\in K$ the expression $|\phi(x)-\phi(z)|$ can be estimated by the sum of three terms:
\begin{enumerate}
    \item $\displaystyle{\int_{B(x,\rho)\cap B(z,\rho)} \left| \sum_{i=0}^{m}a_{i}(x)(y-x)^{i} - \sum_{i=0}^{m}a_{i}(z)(y-z)^{i} \right| \, dy}$\\
    \item $\displaystyle{\int_{B(x,\rho)\setminus B(z,\rho)} \left|u(y)- \sum_{i=0}^{m}a_{i}(x)(y-x)^{i}\right|\, dy}$\\
    \item $\displaystyle{\int_{B(z,\rho)\setminus B(x,\rho)} \left|u(y)- \sum_{i=0}^{m}a_{i}(z)(y-z)^{i}\right|\, dy}$
\end{enumerate}
We will show that each term vanishes as $z\to x$

Firstly, $(1)$ can be estimated by
\[\sum_{i=0}^{m} \int_{B(x,\rho)\cap B(z,\rho)} | a_{i}(x)(y-x)^{i} - a_{i}(z)(y-z)^{i} | \, dy. \]
Since the functions $a_{i}$ are continuous on the compact set $K$, they are bounded by some constant $C_{K}$. We estimate each term in the sum as follows,
\begin{align*}
& \int_{B(x,\rho)\cap B(z,\rho)} | a_{i}(x)(y-x)^{i} - a_{i}(z)(y-z)^{i} | \, dy\\
&\qquad \leq \int_{B(x,\rho)\cap B(z,\rho)} |a_{i}(x)-a_{i}(z)|\cdot |y-x|^{i} + |a_{i}(z)|\cdot ||y-x|^{i}-|y-z|^{i}| \, dy\\
&\qquad \leq 2\rho^{i+1} |a_{i}(x)-a_{i}(z)| + C_{K} \int_{B(x,\rho)\cap B(z,\rho)} ||y-x|^{i}-|y-z|^{i}| \, dy.
\end{align*}
Using continuity of each function $a_{i}$ on $K$, the limit of the first term is $0$ as $x\to z$. If $i=0$ then the second term is identically zero. Otherwise, note that $t\mapsto t^{i}$ is Lipschitz on the interval $|t|\leq \rho$ with Lipschitz constant $i\rho^{i-1}$ by the mean value theorem. Hence, the second term can be estimated by 
\[ \int_{B(x,\rho)\cap B(z,\rho)} i\rho^{i-1}|x-z|\, dy \leq 2i\rho^{i}|x-z|,\]
which clearly tends to $0$ as $x\to z$. This completes the estimate of (1).

Next, we estimate $(2)$ as follows, recalling that $\rho<1$,
\begin{align}
&\int_{B(x,\rho)\setminus B(z,\rho)} |u(y)- \sum_{i=0}^{m}a_{i}(x)(y-x)^{i}|\, dy \nonumber\\
&\qquad \leq \int_{B(x,\rho)\setminus B(z,\rho)}|u(y)|\, dy + \left(\sum_{i=0}^{m}|a_{i}(x)|\rho^{i} \right) \mathcal{L}^1(B(x,\rho)\setminus B(z,\rho)) \nonumber\\
&\qquad \leq \int_{B(x,\rho)\setminus B(z,\rho)}|u(y)|\, dy + C_{K} (m+1) \mathcal{L}^1(B(x,\rho)\setminus B(z,\rho)). \label{e-lastline}
\end{align}
Clearly $\mathcal{L}^1(B(x,\rho)\setminus B(z,\rho))\to 0$ as $z\to x$. Since $u\in L^{1}([a,b])$, it follows that \eqref{e-lastline} tends to $0$ as $z\to x$.

The estimate for $(3)$ is similar to that of (2) with $x$ and $z$ interchanged. This completes the proof of the claim.
\end{proof}
Combining the two claims above completes the proof of Lemma \ref{uniformparameters}.
\end{proof}

Recall Markov's inequality for polynomials \cite{Mar89, Sha04}.


\begin{lemma}[Markov Inequality]\label{Markov}
Let $P$ be a polynomial of degree $n$ and $a < b$. Then
\[ \max_{[a,b]}|P'|\leq \frac{2n^{2}}{b-a} \max_{[a,b]}|P|.\]
\end{lemma}

The following result will establish the first condition necessary in the application of Theorem~\ref{t-HeisWhitLip}.

\begin{proposition}\label{approxWhitney}
Suppose $u\colon [a,b]\to \mathbb{R}$ is absolutely continuous and $u'$ is $(m-1)$-times $L^{1,\omega}$ differentiable almost everywhere. For almost every $x\in (a,b)$, denote the $m$-times $L^{1,\omega}$ derivative of $u$ at $x$ by $P_{u,x}^m(y) = \sum_{i=0}^m\frac{u_i(x)}{i!}(y-x)^i$.

Then for every $\varepsilon>0$, there exists a compact set $K\subset [a,b]$ with $\mathcal{L}^1([a,b]\setminus K) \leq \varepsilon$ such that $(u_{i})_{i=0}^{m}$ is a $C^{m,\omega}$ Whitney field on $K$.
\end{proposition}

\begin{proof} 
Let $\varepsilon>0$. Applying Lemma \ref{uniformparameters} to $u'$, there exists a compact set $K\subset [a,b]$ with $\mathcal{L}^1([a,b]\setminus K)<\varepsilon/2$ and $\rho_{0}>0, C>0$ so that the following holds. For every $x\in K$, $B(x,\rho_{0})\subset [a,b]$ and $u'$ is $m-1$ times $L^{1,\omega}$ differentiable at $x$ with
\[\dashint_{B(x,\rho)} |u'-(P_{u,x}^m)'|\leq C\omega(\rho)\rho^{m-1} \mbox{ for every }0<\rho<\rho_{0}.\]

Suppose $x,y\in K$ with $x\leq y$ and $|x-y|< \rho_{0}$. Denote $P_{y}=P_{u,y}^{m}$ and $P_{x}=P_{u,x}^{m}$. 
Then for $z\in [x,y]$ we estimate as follows, using the fact that $|z-x|< \rho_{0}$, $|z-y|< \rho_{0}$, $P_{y}(y)=u(y)$ and $P_{x}(x)=u(x)$,
\begin{align*}
|P_{y}(z)-P_{x}(z)|
&\leq |P_{y}(z)-u(z)|+|u(z)-P_{x}(z)|\\
&\leq \int_{z}^{y}|P_{y}'-u'|+\int_{x}^{z}|u'-P_{x}'|\\
&\leq 2|z-y|\dashint_{B(y,|z-y|)}|P_{y}'-u'|+2|z-x|\dashint_{B(x,|z-x|)}|u'-P_{x}'|\\
&\leq 2C|z-y|\omega(|z-y|)|z-y|^{m-1}+2C|z-x|\omega(|z-x|)|z-x|^{m-1}\\
&\leq 4C\omega(|x-y|)|x-y|^{m}.
\end{align*}
Applying the Markov inequality (Lemma \ref{Markov}) to the polynomial $q=P_{y}-P_{x}$ gives a larger constant $C$ depending only on $m$ and the previous constant such that, for $z\in [x,y]$,
\[|D^{k}(P_{y})(z)-D^{k}(P_{x})(z)|\leq C\omega(|x-y|)|x-y|^{m-k} \mbox{ for }0\leq k\leq m,\]
where the derivatives are taken with respect to $z$. Evaluating at $z=y$ gives
\begin{equation}\label{smallscales}
|u_{k}(y)-D^{k}(P_{x})(y)|\leq C \omega(|x-y|)|x-y|^{m-k} \mbox{ for }0\leq k\leq m
\end{equation}
for any $x,y\in K$ with $|x-y|< \rho_{0}$

Using the measurability of $u_{k}$ from Remark~\ref{rem44}, apply Lusin's theorem to choose a compact set $\widetilde{K}\subset K$ so that $\mathcal{L}^1([a,b]\setminus \widetilde{K})\leq \varepsilon$ and $u_{k}|_{\widetilde{K}}$ is continuous and hence bounded for every $0\leq k\leq m$. Then the left hand side of \eqref{smallscales} is uniformly bounded above for all $x,y\in \widetilde{K}$ and $0\leq k\leq m$. It follows that, for some larger constant $\widetilde{C}$ depending on $N$, $m$, $K$ and $\widetilde{K}$, we have
\[|u_{k}(y)-D^{k}(P_{x})(y)|\leq \widetilde{C} \omega(|x-y|)|x-y|^{m-k} \mbox{ for }0\leq k\leq m\]
for all $x,y\in \widetilde{K}$ (not only those satisfying $|x-y|< \rho_{0}$). This proves that $(u_{i})_{i=0}^{m}$ is a $C^{m,\omega}$ Whitney field on $\widetilde{K}$.
\end{proof}

We can now conclude our second main theorem. Most of the new work has been done in the above results, and, with these in hand, the proof of this theorem follows nearly the same proof of \cite[Theorem 4.1]{CapPinSpe}. 

\begin{theorem}
\label{t-Lusin}
Let $I\subset\mathbb{R}$ be closed bounded interval, $\omega$ a modulus of continuity, and $\Gamma = (f,g,h):I\to\mathbb{H}$ be a horizontal curve such that $f'$ and $g'$ are $m-1$ times $L^{1,\omega}$ differentiable at almost every point of $I$. Then $\Gamma$ has the $(m,\omega)$-Lusin property. 
\end{theorem}

\begin{proof}
By Lemma \ref{vertL1} we know $h'$ is $(m-1)$ times $L^{1,\omega}$ differentiable almost everywhere. Using Lemma \ref{intL1} it follows that $f,g,h$ are $m$ times $L^{1,\omega}$ differentiable almost everywhere. At almost every $x\in I$ denote the $L^{1,\omega}$ derivative of $f$ by
\[P_{f,x}^m(y) = \sum_{k=0}^m\frac{f_k(x)}{k!}(y-x)^k\]
where the $f_k$ are measurable functions by Remark~\ref{rem44}. Similarly define the $L^{1,\omega}$ derivatives $P_{g,x}^m$ and $P_{h,x}^m$ with coefficients $g_k(x)$ and $h_k(x)$ at almost every point $x$. These coefficients are again measurable functions of $x$.

Fix $\eta>0$. Choose a compact set $K\subset I$ satisfying $\mathcal{L}^{1}(I\setminus K)<\eta$ such that:
\begin{enumerate}
\item the jets $F, G, H$ defined on $K$ by
\[F^k = f_k|_{K},\ G^k = g_k|_{K} \ \text{ and }\ H^k = h_k|_{K}  \ \text{ for } 0\leq k \leq m\]
are Whitney fields of class $C^{m,\omega}$ on $K$.
\item There exist $C,\rho_0>0$ such that for every $0<\delta<\rho_0$ if $a,b\in K$ with $|b-a|<\delta$ then 
\begin{equation}\label{uniformL1}
\dashint_a^b |f'-(T_{a}^{m}F)'| \leq C\omega(|b-a|)(b-a)^{m-1} \quad \mbox{and} \quad  \dashint_a^b |g'-(T_{a}^{m}G)'| \leq C\omega(|b-a|)(b-a)^{m-1} .
\end{equation}
\end{enumerate}
The first property can be obtained via Proposition~\ref{approxWhitney} whereas the second property follows from the almost everywhere $(m-1)$ times $L^{1,\omega}$ differentiability of $f', g'$, the fact that $P_{f,a}^m=T_{a}^{m}F, P_{g,a}^m=T_{a}^{m}G$ and Lemma \ref{uniformparameters} with $u=f', g'$. We now show that the hypotheses of Theorem \ref{t-HeisWhitLip} hold for the jets $F, G, H$ on the compact set $K$.
\smallskip

\emph{Verification of Theorem \ref{t-HeisWhitLip}(1).} This follows directly from the definition of $K$.

\smallskip

\emph{Verification of Theorem \ref{t-HeisWhitLip}(2).} This follows exactly as in \cite[Theorem 4.1]{CapPinSpe}.

\smallskip


\emph{Verification of Theorem \ref{t-HeisWhitLip}(3).} This follows exactly as in \cite[Theorem 4.1]{CapPinSpe} replacing $\varepsilon$ with $C\omega(|a-b|)$.

\smallskip

We have shown that the jets $F, G, H$ satisfy the hypotheses of Theorem~\ref{t-HeisWhitLip} on the compact set $K$. Hence $\Gamma=(F,G,H)$ extends to a $C^{m,\omega}$ horizontal curve $\widetilde{\Gamma}~=~(\widetilde{f}, \widetilde{g}, \widetilde{h})\colon I\to \mathbb{H}$ satisfying
\[\widetilde{f}^{k}|_{K}=F^{k},\quad \widetilde{g}^{k}|_{K}=G^{k},\quad \widetilde{h}^{k}|_{K}=H^{k} \quad \mbox{for }0\leq k\leq m.\]
From the definition of the compact set $K$ and the jets $F, G, H$ we have
\begin{align*}
&\mathcal{L}^{1} \left(\bigcup_{k=0}^{m}\{x\in I: \widetilde{f}^{k}(x)\neq f_{k}(x) \, \mbox{ or } \, \widetilde{g}^{k}(x)\neq g_{k}(x) \, \mbox{ or } \,  \widetilde{h}^{k}(x)\neq h_{k}(x)\}\right)\\
&\quad \leq \mathcal{L}^{1}(I\setminus K)\\
&\quad<\eta.
\end{align*}
This completes the proof of the theorem.
\end{proof}

\section{$C^\infty$ Whitney Extension for Horizontal Curves}
\label{SectionCinfWhitney}

In this section we prove our third result (Theorem \ref{t-Cinfty}), a $C^\infty$ Whitney extension theorem for horizontal curves in the Heisenberg group. Fix the modulus of continuity $\omega(t) = t$. With this particular modulus, the definition of the $\omega$-velocity reads as follows:
$$
V^m_x (\g;a,b) = (b-a)^{2m+2} + (b-a)^{m+1} \int_a^b \left( |(T_aF)'|+ |(T_aG)'| \right)
$$
We will prove the following.

\begin{theorem}
\label{t-Cinfty}
Suppose $K \subseteq \R$ is compact and 
$F=(F^k)_{k=0}^\infty$, $G=(G^k)_{k=0}^\infty$, and $H=(H^k)_{k=0}^\infty$ are collections of continuous, real-valued functions on $K$.
There is a horizontal curve $\Gamma \in C^{\infty}(\R,\R^3)$  
satisfying $D^k \Gamma|_K = (F^k,G^k,H^k)$
for all $k \in \mathbb{N}_0$
if and only if
there are constants $\hat{C}_m > 0$ satisfying the following for every $m \in \mathbb{N}_0$:
\begin{enumerate}
    \item 
    $(F^k)_{k=0}^m$, $(G^k)_{k=0}^m$, and $(H^k)_{k=0}^m$ are Whitney fields of class $C^{m}$ on $K$,
    \item the following holds on $K$:
    $$
        H^m = 2 \sum_{i=0}^{m-1}  \binom{m-1}{i}  \left(F^{m-i}G^i- G^{m-i}F^i \right),
    $$
    \item 
    and, writing $\gamma = (F_0,G_0,H_0)$,
    $$
\left| A^m(\g;a,b) \right| \leq \hat{C}_m V^m_x (\g;a,b)
\quad \text{for all } a,b \in K \text{ with } a<b.
$$
\end{enumerate}
\end{theorem}

\begin{proof}
Necessity follows immediately from Theorem~\ref{t-HeisWhitLip}
since all derivatives of a $C^\infty$ curve are Lipschitz continuous on the compact set $K$.
We will prove sufficiency.
Fix a compact set $K \subset \mathbb{R}$.
Suppose $F=(F^k)_{k=0}^\infty$, $G=(G^k)_{k=0}^\infty$, and $H=(H^k)_{k=0}^\infty$ are collections of continuous, real-valued functions on $K$
such that conditions {\em (1), (2),} and {\em (3)} hold
for all $m \in \mathbb{N}$.
According to Whitney's Extension Theorem for $C^\infty$ functions (Theorem~\ref{classicalWhitneyInf}),
condition {\em (1)} allows us to choose functions $f,g\colon \mathbb{R}\to \mathbb{R}$ of class $C^\infty$ such that
\[
D^mf(x)=F^{m}(x) \mbox{ and }D^mg(x)=G^{m}(x) \mbox{ for every }x\in K \mbox{ and } m \in \mathbb{N}_0.
\]
Fix an open interval $I$ containing $K$ and write $I \setminus K = \bigcup_{i = 1}^\infty (a_i,b_i)$.
For each $m \in \mathbb{N}_0$, use $\overline{C}_m$ to represent the constant in \eqref{e-taylor1} and \eqref{e-taylor2},
and set $\kappa_m = \max \{ \overline{C}_m, \hat{C}_m\}$.

The key difference between the proof of this result and the proof of Theorem~\ref{t-HeisWhitLip} (which may be found in \cite{ZimSpeWhitney}) is the following proposition.
This replaces Proposition~4.2 in \cite{ZimSpeWhitney}.
\begin{proposition}\label{perturb}
There is a strictly decreasing sequence $(c_m)_{m \geq 0}$ with $c_0 \leq 1$ and $c_m \to 0$
satisfying the following.
If $(b_i - a_i) \leq c_m$, then there exist $C^{\infty}$ functions $\phi, \psi : [a_i,b_i]\to \mathbb{R}$ such that
\begin{enumerate}
\item $D^k\phi(a_i)=D^k\phi(b_i)=D^k\psi(a_i)=D^k\psi(b_i)=0$ for all $k \in \mathbb{N}_0$.
\item $\max\{|D^k\phi|, |D^k\psi|\} \leq \sqrt{b_i-a_i}$ for $0\leq k\leq m$ on $[a_i,b_i]$.
\item $H(b_i)-H(a_i)=2\int_{a_i}^{b_i}(f+\phi)'(g+\psi)-(g+\psi)'(f+\phi)$.
\end{enumerate}
\end{proposition}
Note in particular that the bound in {\em (2)} is independent of $m$, while in \cite{ZimSpeWhitney} it was not. 
Note also that we allow for larger and larger regularity of the perturbations as the size of our interval shrinks whereas $m$ was fixed in \cite{ZimSpeWhitney}.
The proof of this proposition is very similar to that of Proposition~4.2 in \cite{ZimSpeWhitney}, so we will leave out many of the details and point out the main differences below.
\begin{proof}[Proof of Proposition \ref{perturb}]
Fix $m \in \mathbb{N}$, and fix $(a,b) = (a_i,b_i)$ for some $i \in \mathbb{N}$.
Write $V^m = V_x^m((f,g,H^0);a,b)$ and
\[
\mathcal{A}:=H^0(b)-H^0(a)-2\int_{a}^{b}(f'g-g'f).
\]
As in \cite{ZimSpeWhitney}, we may conclude 
from condition {\em (3)} of our hypotheses
that
$
|\mathcal{A}| 
\leq \kappa_m V^m
$.
Also, after rewriting $\mathcal{A}$ using integration by parts as in \cite{ZimSpeWhitney},
our goal of constructing $\phi$ and $\psi$ which satisfy Proposition~\ref{perturb}{\em (3)} is equivalent to solving
\begin{equation}
\label{goalint}
4\int_{a}^{b}(\psi f' - \phi g' + \psi \phi')=\mathcal{A}
\end{equation}
subject to conditions {\em (1)} and {\em (2)}.
This will be done as in \cite{ZimSpeWhitney}.
Let us first restate Lemma~4.4 from \cite{ZimSpeWhitney} when $\omega(t) = t$.
\begin{lemma}\label{mollifier}
Suppose $J \subseteq [a,b]$ is a closed interval of length at least $(b-a)/18m^2$.
There is a constant $C_m > 0$ depending only on $m$ and $\mathrm{diam}(K)$ 
and a non-negative $C^{\infty}$ function $\eta:\mathbb{R} \to \mathbb{R}$
such that
\begin{enumerate}
\item $\eta$ vanishes outside of $J$,
\item $\eta \geq 48m^2  (b-a)^{m+1}$ on the middle third of $J$,
\item $\eta' \geq 81 m^2 (b-a)^{m}$ on a subinterval of $J$ with length $\ell(J)/6$,
\item $\eta \leq C_m (b-a)^{m+1}$ 
\item $|D^i\eta| \leq C_m (b-a)$ on $[a,b]$ for $0\leq i\leq m$,
\end{enumerate}
\end{lemma}

As a result, we have the following analogues of Lemmas~4.5 and 4.6 from \cite{ZimSpeWhitney}
(the proofs of which are identical to those of the referenced lemmas).
\begin{lemma}\label{fbig}
Fix $m \in \mathbb{N}_0$.
Suppose
\begin{equation}
\label{fbigassumption}
\int_{a}^{b} |(Tf)'| \geq \max \left( \int_{a}^{b} |(Tg)'|,\ \kappa_m C_m (b-a)^{m+1} \right).
\end{equation}
Then there exists a $C^{\infty}$ map $\psi$ on $[a,b]$ 
satisfying
\begin{enumerate}
\item $D^i\psi(a)=D^i\psi(b)=0$ for $0\leq i\leq m$,
\item $|D^i\psi(x)|\leq \kappa_mC_m    (b-a)$ on $[a,b]$ for $0\leq i\leq m$,
\item $4 \int_{a}^{b} \psi f' =\mathcal{A}$.
\end{enumerate}
\end{lemma}

\begin{lemma}\label{gbig}
Fix $m \in \mathbb{N}_0$.
Suppose
\begin{equation}
\label{gbigassumption}
\int_{a}^{b} |(Tg)'| \geq \max \left( \int_{a}^{b} |(Tf)'|,\ \kappa_m C_m (b-a)^{m+1} \right).
\end{equation}
Then there exists a $C^{\infty}$ map $\phi$ on $[a,b]$ 
satisfying
\begin{enumerate}
\item $D^i\phi(a)=D^i\phi(b)=0$ for $0\leq i\leq m$,
\item $|D^i\phi(x)|\leq \kappa_mC_m    (b-a)$ on $[a,b]$ for $0\leq i\leq m$,
\item $-4 \int_{a}^{b} \phi g' =\mathcal{A}$.
\end{enumerate}
\end{lemma}

\begin{lemma}\label{fgsmall}
Fix $m \in \mathbb{N}_0$.
Suppose
\[\kappa_m C_m (b-a)^{m+1} > \max \left( \int_{a}^{b} |(Tf)'|,\ \int_{a}^{b} |(Tg)'| \right).\]
Then there exist $C^{\infty}$ maps $\phi$ and $\psi$ 
and a constant $C_m' \geq 1$ depending only on $\kappa_m$, $m$ and $\text{diam}(K)$
such that
\begin{enumerate}
\item $D^i\psi(a)=D^i\psi(b)=D^i\phi(a)=D^i\phi(b)=0$ for $0\leq i\leq m$,
\item $\max \{|D^i\psi(x)|, |D^i\phi(x)| \} \leq 6C_m' (b-a)$ on $[a,b]$ for $0\leq i\leq m$,
\item $4 \int_{a}^{b} (\psi f' - \phi g' + \psi \phi') =\mathcal{A}$.
\end{enumerate}
\end{lemma}

We have now reached the step at which our proof of Proposition~\ref{perturb} deviates from that of Proposition~4.2 in \cite{ZimSpeWhitney}.
Set $\widetilde{C}_{-1} = 1$ and $\widetilde{C}_m=\max\{ \kappa_m C_m, 6C_m', \widetilde{C}_{m-1} + 1 \}$ for $m \in \mathbb{N}_0$.
Set $c_m = (1/\widetilde{C}_m)^2$.
Note that $c_m > c_{m+1}$ and $c_m \to 0$ since $\widetilde{C}_m \to \infty$.
Fix $m \in \mathbb{N}_0$ and an interval $(a_i,b_i)$.
Suppose that $(b_i - a_i) \leq c_m$ so that
$\sqrt{b_i - a_i} \leq 1/\widetilde{C}_m$.

If \eqref{fbigassumption} holds on $[a_i,b_i]$, 
then we may choose $\phi \equiv 0$ on $[a_i,b_i]$ and $\psi$ as in Lemma \ref{fbig}
so that
$$
|D^k\psi|\leq \kappa_{m} C_{m} (b_i-a_i) \leq \widetilde{C}_{m} \sqrt{b_i-a_i} \sqrt{b_i-a_i}
\leq \sqrt{b_i-a_i}
$$ 
on $[a_i,b_i]$ for $0\leq k \leq m$.
If \eqref{gbigassumption} holds, then we may use Lemma~\ref{gbig} analogously to obtain the bound on $|D^k \phi|$ and set $\psi \equiv 0$.
If instead \eqref{fbigassumption} and \eqref{gbigassumption} both fail, we may choose $\phi$ and $\psi$ as in Lemma~\ref{fgsmall} to get
$$
\max \{|D^k\psi(x)|, |D^k\phi(x)| \} \leq 6C_{m}' (b_i-a_i) \leq \sqrt{b_i-a_i}
$$ on $[a_i,b_i]$ for $0\leq k\leq m$.
\end{proof}

With this proposition in hand, we are ready to complete
the proof of Theorem~\ref{t-Cinfty}.
Define the curve $\Gamma = (\mathcal{F},\mathcal{G},\mathcal{H}) \colon I\to \mathbb{H}$ as follows:
\[ \Gamma(x) := (F(x),G(x),H(x)) \quad \mbox{if }x \in K\]
and
\[\Gamma(x) := (\mathcal{F}_i(x),\mathcal{G}_i(x),\mathcal{H}_i(x)) \quad \mbox{if }x \in (a_i,b_i) \mbox{ for some }i \geq 1\]
where $\mathcal{F}_i$, $\mathcal{G}_i$, and $\mathcal{H}_i$
are defined as in
Lemma~4.7 from \cite{ZimSpeWhitney}
with ``$2 \widetilde{C} \omega(b_i-a_i)$'' replaced by ``$\sqrt{b_i-a_i}$''.

\begin{proposition}\label{conclusion}
$\Gamma$ is a $C^{\infty}$ horizontal curve in $\mathbb{H}$ with
\[D^k \mathcal{F}(x)=F^{k}(x),\qquad D^k \mathcal{G}(x)=G^{k}(x), \qquad D^k \mathcal{H}(x)=H^{k}(x)\]
 for all $x\in K$ and $k \in \mathbb{N}_0$.
\end{proposition}

\begin{proof}
Clearly the curve $\Gamma$ is $C^{\infty}$ in the subintervals $(a_{i}, b_{i})$. Define maps $\gamma^k$ on $K$ for $k \in \mathbb{N}_0$ by $\gamma^k=(F^k,G^k,H^k)$. 
With this notation, it remains to show that $\Gamma$ is a $C^{\infty}$ horizontal curve and $D^k \Gamma|_{K}=\gamma^{k}$ for $k \in \mathbb{N}_0$.

Fix $k \in \mathbb{N}$ and suppose we have shown that $D^{k-1} \Gamma$ exists on $I$ and $D^{k-1} \Gamma|_K = \gamma^{k-1}$.
Fix $x \in K$ with $x \neq a_i$ for any $i \in \mathbb{N}$ and $x \neq \max  K $. 
As in (6.17) from \cite{ZimSpePinWhitney}, it suffices to prove that 
\begin{equation}
\label{diffquot}
(x_i - x)^{-1} | D^{k-1} \Gamma(x_i) - D^{k-1} \Gamma(x) - (x_i-x)\gamma^k(x) |
\end{equation}
vanishes in the limit for any decreasing sequence $(x_i) \subset I \setminus K$ with $x_i \to x$.
Choose such a sequence
and note that, for each $i \in \mathbb{N}$, we have $x_i \in (a_{j_i},b_{j_i})$ for some $j_i\in \mathbb{N}$.
Since $x$ is not the left endpoint of any of these intervals, it must be true that $(b_{j_i} - a_{j_i}) \to 0$ as $i \to \infty$.
After passing to a subsequence, we may assume that $(b_{j_i} - a_{j_i}) \leq c_k$ for all $i$.

We may then bound
\eqref{diffquot} by
\begin{align}
(x_i - x)^{-1} 
| &D^{k-1} \Gamma(x_i) - D^{k-1} \Gamma(a_{j_i}) - (x_i-a_{j_i})\gamma^k(a_{j_i}) | \label{Bound1} \\
&+
(x_i - x)^{-1} 
| (x_i - a_{j_i}) \gamma^k(a_{j_i}) - (x_i-a_{j_i})\gamma^k(x) | \label{Bound2} \\
&+
(x_i - x)^{-1} 
| \gamma^{k-1}(a_{j_i}) - \gamma^{k-1}(x) - (a_{j_i} - x)\gamma^k(x) | \label{Bound3} 
\end{align}
and argue exactly as in \cite{ZimSpeWhitney}
to conclude that 
the right-hand derivative of $D^{k-1} \Gamma$ at $x$ is $\gamma^k(x)$.
We argue similarly using increasing sequences to finally prove that 
$D^k \Gamma$ exists on $I$, and $D^k \Gamma|_K = \gamma^k$.
\end{proof}
This completes the proof of the theorem.
\end{proof}

\section{A $C^\infty$ Lusin Approximation Theorem}
\label{SectionCinfLusin}

In this section we prove our final result (Theorem \ref{tt-Lusin}), a $C^\infty$ Lusin approximation theorem for horizontal curves. 

\begin{theorem}
\label{tt-Lusin}
Let $I\subset\mathbb{R}$ be a closed bounded interval and $\Gamma = (f,g,h):I\to\mathbb{H}$ be a horizontal curve such that for every $m\in\mathbb{N}$, $f'$ and $g'$ are $m$ times $L^{1}$ differentiable at almost every point of $I$. Then $\Gamma$ has the $\infty$-Lusin property. 
\end{theorem}

\begin{proof}
Let $\omega(t)=t$.
By Remark \ref{gain} and Lemma \ref{vertL1}, for every $m\in\mathbb{N}$ the maps $f'$, $g'$, $h'$ are $m$ times $L^{1,\omega}$ differentiable at almost every point of $I$. Hence, by Lemma \ref{intL1}, for every $m\in \mathbb{N}$ the maps $f,g,h$ are $m$ times $L^{1,\omega}$ differentiable almost everywhere. 
For almost every $x\in I$ and every $m\in \mathbb{N}$, denote the $L^{1,\omega}$ derivative of $f$ by
\[P_{f,x}^m(y) = \sum_{k=0}^m\frac{f_k(x)}{k!}(y-x)^k\]
where the $f_k$ are measurable functions by Remark~\ref{rem44}. Note that the maps $f_{k}$ are independent of $m$ by Remark \ref{rem3}.
Similarly, for almost every $x\in I$ and every $m\in \mathbb{N}$, denote the $L^{1,\omega}$ derivatives of order $m$ of $g$ and $h$ by $P_{g,x}^m$ and $P_{h,x}^m$. As before, these have coefficients $g_k(x)$ and $h_{k}(x)$ which are measurable functions of $x$.
Fix $\epsilon>0$. Proceeding as in Theorem \ref{t-Lusin}, for any $m\in \mathbb{N}$ choose a compact set $K_m\subset [a,b]$ satisfying $\mathcal{L}^1([a,b]\setminus K_m)< \epsilon / 2^m$ so that 
\begin{enumerate}
\item the jets $F_{m}, G_{m}, H_{m}$ defined on $K_m$ by
\[F_{m}^k = f_k|_{K_{m}},\ G_{m}^k = g_k|_{K_{m}} \ \text{ and }\ H_{m}^k = h_k|_{K_{m}}  \ \text{ for } 0\leq k \leq m\]
are Whitney fields of class $C^{m}$ on $K_m$.
\item There exist $C,\rho_0>0$ such that for every $0<\delta<\rho_0$ if $a,b\in K_{m}$ with $|b-a|<\delta$,
\[\dashint_a^b |f'-(T_{a}^{m}F_{m})'| \leq C\omega(|b-a|)(b-a)^{m-1} \quad \mbox{and} \quad  \dashint_a^b |g'-(T_{a}^{m}G_{m})'| \leq C\omega(|b-a|)(b-a)^{m-1} .\]
\end{enumerate}
Proceeding as in the proof of Theorem \ref{t-Lusin}, it is easy to see that $F=(F^k)_{k=0}^\infty$, $G=(G^k)_{k=0}^\infty$, and $H=(H^k)_{k=0}^\infty$ are well-defined and that they satisfy the assumption of Theorem \ref{t-Cinfty} on the compact set $K=\bigcap_{m=1}^\infty K_m$.  Hence $\Gamma=(F,G,H)$ extends to a $C^{\infty}$ horizontal curve $\widetilde{\Gamma}~=~(\widetilde{f}, \widetilde{g}, \widetilde{h})\colon I\to \mathbb{H}$ satisfying for any $k\in \mathbb{N}_{0}$
\[\widetilde{f}^{k}|_{K}=F^{k},\quad \widetilde{g}^{k}|_{K}=G^{k},\quad \widetilde{h}^{k}|_{K}=H^{k}. \]
From the definition of the compact set $K$ and the jets $F, G, H$ we have
\begin{align*}
&\mathcal{L}^{1} \left(\bigcup_{k=0}^{\infty}\{x\in I: \widetilde{f}^{k}(x)\neq f_{k}(x) \, \mbox{ or } \, \widetilde{g}^{k}(x)\neq g_{k}(x) \, \mbox{ or } \,  \widetilde{h}^{k}(x)\neq h_{k}(x)\}\right)\\
&\quad \leq \sum_{m=1}^{\infty}\mathcal{L}^{1}(I\setminus K_m)\\
&\quad<\sum_{m=1}^{\infty}\frac{\varepsilon}{2^m}=\varepsilon.
\end{align*}
This completes the proof of the theorem.
\end{proof}

\bibliography{zimbib} 
\bibliographystyle{acm}

\end{document}